\documentclass{amsart}

\newtheorem{theorem}{Theorem}[section]
\newtheorem{corollary}[theorem]{Corollary}
\newtheorem{proposition}[theorem]{Proposition}
\newtheorem{lemma}[theorem]{Lemma}

\theoremstyle{definition}

\usepackage{mathrsfs}

\begin{document}

\title[Transport equations with fractal noise]{Transport equations with fractal noise -\\existence, uniqueness and regularity\\of the solution}


\author{ Elena Issoglio }

\address{E. Issoglio: Institut f\"ur Stochastik, Friedrich-Schiller Universit\"at, Jena, Germany; elena.issoglio@uni-jena.de}
\address{School of Mathematics,  The University of Manchester, Manchester, United Kingdom; elena.issoglio@manchester.co.uk}



\begin{abstract}
The main result of the present paper is a statement on existence, uniqueness and regularity for mild solutions to a parabolic transport diffusion type equation that involves a non-smooth coefficient. We investigate related Cauchy problems on bounded smooth domains with Dirichlet boundary conditions by means of semigroup theory and fixed point arguments. Main ingredients are the definition of a product of a function and a (not too irregular) distribution as well as a corresponding norm estimate. As an application, transport stochastic partial differential equations driven by fractional Brownian noises are considered in the pathwise sense.
\end{abstract}





\maketitle


\newcommand{\supp}{\operatorname{supp}}

\section{Introduction}

We consider the following transport equation on a domain $D\subset\mathbb{R}^d$ with initial and Dirichlet boundary conditions:
\begin{equation}\label{SPDE iniziale Dirichlet}
\left\{\begin{array}{lcr}
\frac{\partial u}{\partial t}(t,x)=  \Delta u (t,x) + \langle \nabla u, \nabla Z\rangle (t,x) ,&\quad& t   \in (0,T], x\in D\\
u(t,x)=0, &&  t\in (0,T], x\in \partial D\\
u(0,x)=u_0(x), && x\in D
\end{array}\right.
\end{equation}
where $D$ is a bounded open set of $\mathbb{R}^d$ with $C^\infty$ boundary, $u_0$ is a given function in some appropriate space, $Z$ is a given non-differentiable function on $\mathbb{R}^d$ and the derivative is taken in the distributional sense.  The gradient $\nabla$ as well as the Laplacian $\Delta$ refer to the space variables. The precise definition of the product $\langle\nabla u, \nabla Z\rangle (t,x)$ will be given below, and it is set by use of the Fourier transform.

The aim of this paper is to give a meaning to the formal problem (\ref{SPDE iniziale Dirichlet}) and to investigate existence, uniqueness and regularity of corresponding solutions. \\
We rewrite problem  (\ref{SPDE iniziale Dirichlet}) in the abstract Cauchy setting, namely we interpret all mappings as functions  of time $t$ taking values  in some suitable function space $X$ (real function space on $\mathbb{R}^d$, our choice will be specified later). Set $\underbar {\textit{u}} :[0,T]\to X, \, t\mapsto \underbar {\textit{u}}(t)\in X$ and $(\underbar {\textit{u}}(t))(\cdot):=u(t,\cdot)$. The Dirichlet initial value problem becomes the following abstract Cauchy problem
\begin{equation}\label{SPDE Cauchy prb}
\left\{\begin{array}{lcr}
\frac{d}{dt}\underbar {\textit{u}}=  \Delta_D \underbar {\textit{u}} + \langle \nabla \underbar {\textit{u}}, \nabla Z\rangle, &\quad&     t\in(0,T]\\
\underbar {\textit{u}}=u_0, &&  t=0
\end{array}\right.
\end{equation}
where $\Delta_D$ stands for the Dirichlet-Laplace operator.\\
Note that we need some care to give an appropriate definition for the product term $\langle \nabla \underbar {\textit{u}}, \nabla Z\rangle$: In the cases we consider the components of  $\nabla Z$ will be distributions.  This is not covered by results in the standard literature for partial differential equations (PDEs) (see for instance \cite{e1, l2}). We use a priori estimates on this product which lead to optimal regularity results. To our knowledge, this has not been considered anywhere else.

There is a rich literature regarding stochastic PDEs (SPDEs) (see for instance \cite{c1, d1, h2} and references therein). In these references the noise is assumed to be of Brownian (or semimartingale) type.\\
There are also results on SPDEs involving fractional Brownian (or general non-semimartingale) type noises (see for instance \cite{d3, g1, h1, h3,  m1, t1} ) but it seems that there are few results on transport diffusion equations with random non-smooth drift of the form (\ref{SPDE Cauchy prb}).

To our knowledge, the only study regarding this problem is  due to Russo and Trutnau \cite{r2} where they investigate a stochastic equation like (\ref{stochastic PDE}) (which is the stochastic analog of (\ref{SPDE iniziale Dirichlet})) but  in  space dimension one. The authors proceed by freezing the realization of the noise for each $\omega$ and overcome the problem of defining the product between a function and a distribution by means of a probabilistic representation: They express the parabolic PDE probabilistically through the associated diffusion which is the solution of a stochastic differential equation with generalized drift.

In the present paper Fourier analysis is used to define \emph{pointwise} products that work for any space dimension (see Proposition \ref{proposition maggiorazione prodotto puntuale per uZ}). The product itself will be a distribution.\\
We proceed as follows: In Section \ref{section preliminaries}, after having introduced the framework and the notion of (mild) solution, we define an integral operator $I$.
The product together with the action of the semigroup and an integration with respect to time will define the integral operator.\\
In Section \ref{section main result} we first collect some useful a priori estimates and bounds,  then we state  the key Theorem \ref{teorema per I in spazi di holder (rho)} dealing with the mapping property of the integral operator in the spaces $C^\gamma([0,T]; \tilde H^{1+\delta}(D))$ and finally we state the main result in Theorem \ref{teo esistenza e unicita globale}. By a contraction argument and under suitable conditions on the papameters $\gamma,\delta>0$, on the noise and on the initial condition we find a unique (mild) solution for (\ref{SPDE iniziale Dirichlet}) in the above-mentioned space. Of interest is the fact that the solution is actually a function, even though we make use of fractional Sobolev spaces of negative index (spaces of distributions) while proving  the desired result.\\
In Section \ref{section applications} we conclude the paper  presenting some applications to stochastic PDEs. We are namely able to solve a class of SPDEs where the noise is, for instance, a temporally homogeneous fractional Brownian field. Moreover combining it with a results of Hinz and Z\"ahle \cite{h1} we can treat the more general (stochastic) transport equation of the form
\begin{equation*}
\left\{\begin{array}{lr}
\frac{\partial u}{\partial t}(t,x)=  \Delta u (t,x) + \langle \nabla u, \nabla Z\rangle (t,x) + \langle F, \frac{\partial}{\partial t}\nabla V \rangle (t,x),& t   \in (0,T], x\in D\\
u(t,x)=0, &  t\in (0,T], x\in \partial D\\
u(0,x)=u_0(x), & x\in D
\end{array}\right.
\end{equation*}
where $F$ is a given vector and $V=V(t,x)$ is a given non-differentiable function.
Throughout the whole paper $c$ denotes a finite positive constant whose exact value is not important and may change from line to line.

\section{Preliminaries}\label{section preliminaries}

\subsection{Framework}\label{subsection framework}
Recall the definition of fractional Sobolev spaces (Bessel potential spaces) on $\mathbb{R}^d$. For $\alpha\in\mathbb{R}$ and $1<p<\infty$ set
\[
H^\alpha_p(\mathbb R^d; \mathbb C):=\left\{ f\in \mathcal{S}'(\mathbb{R}^d; \mathbb C) : ((1+\vert\xi\vert^2)^{\frac{\alpha}{2}} \hat f)^\vee \in L_p(\mathbb{R}^d,\mathbb C) \right\},
\]
equipped with the norm $ \|f\vert H_p^\alpha(\mathbb{R}^d;\mathbb C)\|=\|((1+\vert\xi\vert^2)^{\frac{\alpha}{2}} \hat f)^\vee\vert L_p(\mathbb{R}^d;\mathbb C)\|$, where $\hat f$ stands for the Fourier transform of $f$ on $\mathbb{R}^d$ and $(\cdot)^\vee$ denotes the inverse Fourier transform. We are interested only in real valued distributions (and functions) so we follow \cite{r1} and define
$\mathcal{S}'(\mathbb{R}^d;\mathbb{R}):=\{f\in\mathcal{S}'(\mathbb{R}^d;\mathbb C): \bar f=f \}$ where $\bar f$ is defined by $\bar f(\phi)=f(\bar \phi)$ for all $\phi\in\mathcal{S}(\mathbb{R}^d,\mathbb C)$. For $1<p<\infty$ and $\alpha\in \mathbb{R}$ we define $H_p^\alpha(\mathbb{R}^d;\mathbb{R}):= H_p^\alpha(\mathbb{R}^d;\mathbb C)\cap \mathcal S'(\mathbb{R}^d;\mathbb{R})$.\\
For simplicity of notation we omit the writing of the codomain when it is $\mathbb{R}$.\\
The corresponding Sobolev spaces on $D$, suitable for our purposes, are defined for all $\alpha>-\frac{1}{2}$ as
\[
\tilde H^\alpha_p(D):=\left\{ f\in H^\alpha_p(\mathbb{R}^d): \supp(f)\subset \bar D \right\}
\]
equipped with the norm  in $H^\alpha_p(\mathbb{R}^d)$.  Observe that if $\alpha=0$ then the space $\tilde H^0_p(D)$ is simply $L_p(D)$. Such spaces are embedded in each other in the following way: For all $\alpha>\beta$, $ H_p^\alpha (\mathbb{R}^d)\subset H_p^\beta(\mathbb{R}^d)$. An analogous relation holds for the spaces on domain $D$ for all $\alpha>\beta>-\frac{1}{2}$. \\
We omit the subscript index  $p$ if $p=2$. In this case the norm in $H^\alpha(\mathbb{R}^d)$ is denoted by  $\|\cdot\|_\alpha$. Moreover when we have a vector (like $\nabla Z$) we write $\nabla Z \in  H_p^{\alpha}(\mathbb{R}^d)$ (and similarly for spaces on $D$) to intend that every component of the vector $\nabla Z$ belongs to such space. The norm of a $d$-dimensional vector in the space $( H_p^{\alpha}(\mathbb{R}^d))^d$ is defined as the square root of the sum of the squared norm of each component in $ H_p^{\alpha}(\mathbb{R}^d)$. For simplicity we will indicate it with the same notation.

Consider now the Dirichlet Laplacian $\Delta_D$ as the infinitesimal generator of the Dirichlet heat semigroup acting on $L_2(D)$ (see e.~g.~\cite{v1} Section 4.1, \cite{e1} Section 7.4.3). Throughout the whole paper we will indicate it with $\Delta_D=-A$. More precisely $-A$ generates a compact $\mathscr{C}_0$ semigroup of contractions $(P_t)_{t\geq 0}$ in $L_2(D)$ (see \cite{v1}, Theorem 7.2.5).  The semigroup is of negative type and symmetric. Moreover it is known that if the semigroup is contractive and symmetric it is also analytic (see \cite{d2}, Theorem 1.4.1  or \cite{s1}, Chapter III), thus one can define fractional powers of $A$ of any order (see for instance \cite{p1}).\\
It can be shown (see \cite{t4} equations (27.50) and (27.51) or \cite{t3} Section 4.9.2) that for all $\gamma,\alpha\in\mathbb{R}$ such that $-\frac{1}{2}<\gamma,\gamma-\alpha<\frac{3}{2}$ the fractional power $  A^{\frac{\alpha}{2}} $ maps isomorphically $ \tilde H^{\gamma}(D)$ onto $  \tilde H^{\gamma-\alpha}(D)$, hence there exist $c_1>0$ and $c_2>0$ such that for all $f\in \tilde H^\gamma(D)$
\begin{equation}\label{bound for the isomorphic spaces H}
\left\|A^{\frac{\alpha}{2}} f \right\|_{\gamma-\alpha}\leq c_1\left\|f\right\|_{\gamma} \leq c_2 \left\|A^{\frac{\alpha}{2}} f \right\|_{\gamma-\alpha}.
\end{equation}
Furthermore one can prove that  $\mathcal{D}(A^{\frac{\alpha}{2}})=\tilde H^\alpha(D)$ for all $0<\alpha<\frac{3}{2}, \alpha\neq \frac{1}{2}$ (for details see \cite{t3}).\\
We recall now a standard result on semigroups, for a proof we refer to \cite{p1}  Theorem II.6.13 or \cite{v1} Theorem 7.7.2.
\begin{theorem}\label{theor pazy}
 Let $-\mathcal{A}$ be the infinitesimal generator of an analytic semigroup $T_t$ on a Banach space $(X,\|\cdot \|_X)$.  If for each $t\geq 0$ holds $\|T_t\|\leq Me^{-\omega t}$ with $M\geq 1$ and $\omega >0$ then
\begin{itemize}
 \item [(a)] $T_t:X\to \mathcal{D}(\mathcal{A}^\alpha)$  for every $t>0, \alpha\geq 0$;
 \item [(b)] for every $\alpha\geq0$ and for every $x\in \mathcal{D} (\mathcal{A}^\alpha)$,  $ T_t\mathcal{A}^\alpha x= \mathcal A^\alpha T_t x$;
 \item [(c)] for every $t>0$ and for every $\alpha\geq0$ the operator $\mathcal{A}^\alpha T_t$ is bounded and linear and there exist constants $M_\alpha$ (which depends only on $\alpha$) and $\theta\in(0,\omega)$ such that
\[
\|\mathcal{A}^\alpha T_t  \|_{\mathscr L(X)} \leq M_\alpha e^{-\theta t} t^{-\alpha};
\]
 \item [(d)] for each $0<\alpha\leq1$ there exists $C_\alpha>0$ such that $\forall t>0$ and for each $x\in \mathcal D(\mathcal{A}^\alpha)$ we have
\[
\|T_t x-x\|_X\leq C_\alpha t^\alpha \|\mathcal{A}^\alpha x\|_X.
\]
\end{itemize}
\end{theorem}

As a consequence of this theorem and of relation (\ref{bound for the isomorphic spaces H}) we get the following result.

\begin{corollary}\label{propr per Pt in sobolev frazionario}
Let $(P_t)_{t\geq0}$ be the  Dirichlet heat semigroup on $L_2(D)$. Then for all positive $t$ and for any $-\frac{1}{2}<\rho,\gamma,\rho+\gamma<\frac{3}{2}$we have
\begin{equation}
 P_t:\tilde H^\gamma (D)\to \tilde H^{\rho+\gamma}(D).
\end{equation}
In particular if $f\in  \tilde H^\gamma (D)$ then $\supp (P_tf)\subset \bar D$.
\end{corollary}
\begin{proof}
Consider first the case when $\gamma>0$. Let $f\in \tilde H^{\gamma}(D)$ so by (\ref{bound for the isomorphic spaces H}) we have  $g:=A^{\frac{\gamma}{2}} f\in L_2(D)$. We write $P_tf=P_tA^{-\frac{\gamma}{2}}A^{\frac{\gamma}{2}} f=P_tA^{-\frac{\gamma}{2}} g=A^{-\frac{\gamma}{2}}P_t g$ and by Theorem \ref{theor pazy} (a) we know that $P_t g\in \mathcal{D}(A^\rho)$ for any $\rho\geq 0$. Moreover recall that $\mathcal{D}(A^2\rho)=\tilde H^{\rho}(D)$ for all $0\leq \rho<\frac{3}{2}, \rho\neq \frac{1}{2} $, so for this choice of $\rho$ and  using (\ref{bound for the isomorphic spaces H}) we get $P_tf=A^{-\frac{\gamma}{2}}P_t g \in \tilde H^{\rho+\gamma}(D)$. Observe that this fact is true also if $\rho=\frac{1}{2}$ since $\tilde H^{\rho+\gamma}(D)\subset \tilde H^{\frac{1}{2}+\gamma}(D)$ for all $\rho>\frac{1}{2}$. \\
The case when $\gamma<0$ is proven in the same way, simply write $A^{-\frac{\gamma}{2}}A^{\frac{\gamma}{2}} P_t f$ instead of $P_tA^{-\frac{\gamma}{2}}A^{\frac{\gamma}{2}} f$.
\end{proof}

\subsection{Mild solutions}
A function $\underbar {\textit{u}}$ is a \emph{mild solution} of (\ref{SPDE Cauchy prb}) if it satisfies the following integral equation
\begin{equation}\label{eq mild solution}
 \underbar {\textit{u}}(t)=P_t u_0+\int_0^t P_{t-r} \langle \nabla \underbar {\textit{u}}(r),\nabla Z\rangle \, \mathrm{d} r
\end{equation}
To give a formal meaning to the product $\langle\cdot,\cdot \rangle$ we make use of the so called \emph{paraproduct}, see e.~g.~\cite{r1}. We shortly recall the definition and some useful properties.\\
Suppose we are given $f\in \mathcal{S} ' (\mathbb{R}^d) $. Choose a function $\psi\in \mathcal{S}(\mathbb{R}^d)$ such that $0\leq \psi(x)\leq 1$ for every $x\in \mathbb{R}^d$, $\psi(x)=1$ if $\vert x\vert\leq 1$ and $\psi(x)=0$ if $\vert x\vert\geq \frac{3}{2}$. Then consider the following approximation of $f$
\[
S^j f (x):= \left(\psi\left(\frac{\xi}{2^j}\right)\hat f\right)^\vee (x)
\]
that is in fact the convolution of $f$ with a smoothing function. This approximation is used  to define the product of two distributions $fg$ as follows:
\[
fg:=\lim _{j\to \infty} S^jfS^jg
\]
if the limit exists in $\mathcal{S}'(\mathbb{R}^d)$. The convergence in the case we are interested in is part of the assertion below (see \cite{h1} appendix C.4, \cite{r1} Theorem 4.4.3/1).
\begin{lemma}\label{lemma maggiorazione pointwise multiplication}
 Let $1<p,q<\infty$ and $0<\beta<\delta$ and assume that $q>\max(p, \frac{d}{\delta})$. Then for every $f\in H^{\delta}_p(\mathbb{R}^d)$ and $g\in H^{-\beta}_q(\mathbb{R}^d)$ we have
\begin{equation}\label{eq maggiorazione paraprodotto}
 \| fg\vert { H_p^{-\beta}(\mathbb{R}^d)}\|\leq c \| f\vert {H_p^{\delta}(\mathbb{R}^d)}\|\cdot \| g\vert { H_q^{-\beta}(\mathbb{R}^d)}\|.
\end{equation}
\end{lemma}
The following Lemma regarding a \emph{locality-preserving} property will be used while shifting the properties of the product $fg$ from the whole $\mathbb{R}^d$ to the domain $D$. For the proof see \cite{r1} Lemma 4.2.
\begin{lemma}\label{lemma supporto pointwise multiplication}
 If $f,g\in \mathcal{S}' (\mathbb{R}^d)$ and $\supp( f) \subset \bar D$ then also $\supp ( fg) \subset \bar D$.
\end{lemma}
Our aim now is to apply such product to $\nabla \underbar {\textit{u}}(s)$ and $\nabla Z$. We will denote by $\langle\cdot,\cdot\rangle$ the pointwise product combined with the scalar product in $\mathbb{R}^d$.

\begin{proposition}\label{proposition maggiorazione prodotto puntuale per uZ}
 Let $\underbar {\textit{u}}(s)\in \tilde H_p^{1+\delta}(D)$, $Z\in H_q^{1-\beta}(\mathbb{R}^d)$ for $1<p,q<\infty$, $q>\max(p,\frac{d}{\delta})$, $0<\beta<\frac{1}{2}$ and $\beta<\delta$. Then the pointwise multiplication $\langle \nabla \underbar {\textit{u}}(s),\nabla Z\rangle$ is well defined, it belongs to the space $\tilde H_p^{-\beta}(D)$ and we have the following bound
\begin{equation}\label{eq bound pointwise product in bar spaces}
 \|\langle \nabla \underbar {\textit{u}}(s),\nabla Z\rangle  \vert { \tilde H_p^{-\beta}(D)}\|\leq c \| \nabla \underbar {\textit{u}}(s) \vert {\tilde H_p^{\delta}(D)}\|\cdot \| \nabla Z \vert { H_q^{-\beta}(\mathbb{R}^d)}\|.
\end{equation}
\end{proposition}

\begin{proof}
The idea is to apply first Lemma \ref{lemma maggiorazione pointwise multiplication}  to define the product as an element of $ H_p^{-\beta}(\mathbb{R}^d)$ and then restrict it to $\tilde  H_p^{-\beta}(D)$ with the help of Lemma \ref{lemma supporto pointwise multiplication}.\\
Let $f=\nabla \underbar {\textit{u}}(s)$ and $g=\nabla Z$. We should check the conditions in Lemma \ref{lemma maggiorazione pointwise multiplication}. Clearly $g\in H^{-\beta}_q(\mathbb{R}^d) $ because $Z\in H^{1-\beta}_q(\mathbb{R}^d)$ and it is easy to show that $(\nabla)_i$ is bounded form $ H^\gamma(\mathbb{R}^d)$ to $H^{\gamma-1}(\mathbb{R}^d)$ for every $\gamma\in \mathbb{R}$ and for all $i=1,\ldots, d$. The fact that $f\in H^{\delta}_p(\mathbb{R}^d)$ is also clear since $ \tilde H^{1+\delta}_p(D)\subset H^{1+\delta}_p(\mathbb{R}^d)$. \\
Denote $m(s):=\langle \nabla \underbar {\textit{u}}(s),\nabla Z\rangle \in H_p^{-\beta}(\mathbb{R}^d) $ and by Lemma \ref{lemma maggiorazione pointwise multiplication} we get
\[
  \| m(s)\vert { H_p^{-\beta}(\mathbb{R}^d)}\|\leq c \|\nabla \underbar {\textit{u}}(s)  \vert {H_p^{\delta}(\mathbb{R}^d)}\|\cdot \| \nabla Z \vert { H_q^{-\beta}(\mathbb{R}^d)}\|<\infty.
\]
Since $\supp \underbar {\textit{u}}(s) \subset \bar D$ then $\supp \nabla \underbar {\textit{u}}(s) \subset \bar D$ and so by Lemma \ref{lemma supporto pointwise multiplication} it follows $\supp m(s)\subset \bar D$ and so $m(s)\in \tilde H_p^{-\beta}(D)$ since $\beta<\frac{1}{2}$. Moreover,
\begin{align*}
  & \|\langle \nabla \underbar {\textit{u}}(s), \nabla Z\rangle  \vert { \tilde H_p^{-\beta}(D)}\|= \|\langle \nabla \underbar {\textit{u}}(s),\nabla Z\rangle  \vert { H_p^{-\beta}(\mathbb{R}^d)}\| \nonumber \\
 \leq  c \| \nabla \underbar {\textit{u}}(s) \vert& {H_p^{\delta}(\mathbb{R}^d)}\|\cdot \| \nabla Z \vert H_q^{-\beta}(\mathbb{R}^d) \|  = c \| \nabla \underbar {\textit{u}}(s) \vert {\tilde H_p^{\delta}(D)}\|\cdot \| \nabla Z \vert H_q^{-\beta}(\mathbb{R}^d) \|.
\end{align*}
\end{proof}

 The notion of mild solution is now formalized. Next we check the convergence of the integral, so for any fixed $\underbar {\textit{u}}(r)\in \tilde H^{1+\delta}(D)$ define the integral operator $I$  by
\begin{equation}\label{def integral operator}
I_{t}(\underbar {\textit{u}}):=\int_0^t P_{t-r} \langle \nabla \underbar {\textit{u}}(r),\nabla Z\rangle \, \mathrm{d} r
\end{equation}
for any $t\in [0,T]$. We consider this operator acting on the H\"older space $C^\gamma([0,T];X)$ into itself (this mapping property will be proven later, see Theorem \ref{teorema per I in spazi di holder (rho)}) for some suitable $\gamma$ and for some infinite dimensional Banach space $X$. The H\"older space is defined as
\[
C^\gamma([0,T];X):=\left\{h:[0,T]\to X \text{ s.t. } \|h\|_{\gamma, X}<\infty \right\}
\]
where
\[
\|h\|_{\gamma, X}:=\sup_{t\in[0,T]}\|h(t)\|_X +
\sup_{s<t\in[0,T]}\frac{\|h(t)-h(s)\|_X}{(t-s)^\gamma} .
\]
When $X=\tilde H^{1+\delta}(D)$ the norm will be indicated by $\|\cdot\|_{\gamma,1+\delta}$. Next we introduce a family of equivalent norms $\|\cdot \|_{\gamma,X}^{(\rho)} , \rho\geq 1 $ defined by
\[
\|f\|_{\gamma,X}^{(\rho)} := \sup_{0\leq t\leq T} e^{-\rho t}\left(\|f(t)\|_{X} +\sup_{0\leq s<t} \frac{\|f(t)-f(s)\|_X}{(t-s)^\gamma} \right).
\]


\section{The main result}\label{section main result}
In this section we prove the contractivity of the operator $I$ in the H\"older space $C^\gamma([0,T];\tilde H^{1+\delta}(D))$.

\subsection{Mapping property of $I$}
Recall that $m(r):= \langle \nabla  \underbar {\textit{u}}(r), \nabla Z\rangle$ for all $0\leq r\leq T$.

\begin{proposition}\label{proposition m(r) in spazi di sobolev}

Let  $0<\beta<\frac{1}{2}$ and $\beta<\delta$ and fix a function $Z\in H_q^{1-\beta}(\mathbb{R}^d)$ for some $q>\max(2,\frac{d}{\delta})$. Then for all $0\leq r\leq t \leq T$ and $\underbar {\textit{u}}(t)\in \tilde H^{1+\delta}(D)$ we have

\begin{enumerate}
 \item [(1)] $\|m(r)\vert {\tilde H^{-\beta}(D)}\| \leq c\|\underbar {\textit{u}}(r)|\tilde H^{1+\delta}(D)\|$;\\
 \item [(2)] $\|m(t)-m(r)\vert {\tilde H^{-\beta}(D)}\| \leq c\|\underbar {\textit{u}}(t)-\underbar {\textit{u}}(r)|\tilde H^{1+\delta}(D)\|$.
\end{enumerate}
\end{proposition}

\begin{proof}
To see (1), observe that by definition $\nabla \underbar {\textit{u}}(r)\in \tilde H^\delta(D)$ means that $ \nabla \underbar {\textit{u}}(r)\in  H^\delta(\mathbb{R}^d)$ and $\supp(\nabla \underbar {\textit{u}}(r)) \subset \bar D$. Also $(\nabla)_j : H^{1+\delta}(\mathbb{R}^d)\to H^\delta(\mathbb{R}^d)$ is bounded for  all $\delta$, i.e. for all $f\in H^{1+\delta}(\mathbb{R}^d) $ there exists $c>0$ such that $\|\nabla f \|_\delta \leq c \| f \|_{1+\delta}$.
These results combined with Proposition \ref{proposition maggiorazione prodotto puntuale per uZ} (where $p=2$) lead to (1).\\
Now we prove (2). Since $\tilde H^{-\beta}(D)$ is a linear space then $m(t)-m(r)\in \tilde H^{-\beta}(D)$. The pointwise product and the operator $\nabla$ are linear so we can write $m(t)-m(r)= \langle \nabla \underbar {\textit{u}}(t)- \nabla \underbar {\textit{u}}(r), \nabla Z\rangle  =  \langle \nabla(  \underbar {\textit{u}}(t)-  \underbar {\textit{u}}(r)), \nabla Z\rangle$. Clearly $\underbar {\textit{u}}(t)-  \underbar {\textit{u}}(r)$ is an element of $ \tilde H^{\delta}(D)\subset   H^{\delta}(\mathbb{R}^d)$ so we proceed in the same way as for (1) and we get the wanted result.
\end{proof}

\begin{proposition}\label{propr su P_t}
Let $0<\beta<\delta<\frac{1}{2}$ and $w\in \tilde H^{-\beta}(D)$. Then $P_t w\in \tilde H^{1+\delta}(D)$ for any $t>0$ and moreover there exists a positive constant $c$ such that
\begin{equation} \label{eq bound per Pt in sobolev frazionario}
\left\|P_t w\right\|_{1+\delta}\leq c\left\|w\right\|_{-\beta} t^{-\frac{1+\delta+\beta}{2}}.
\end{equation}
\end{proposition}
\begin{proof}
Let $w\in \tilde H^{-\beta}(D)$. By (\ref{bound for the isomorphic spaces H}) we have
\begin{align*}
 \|P_t w\|_{1+\delta}& \leq c\| A^{\frac{1+\delta}{2}} P_t w \|_{0} = c\| A^{\frac{1+\delta}{2}}A^{\frac{\beta}{2}}A^{-\frac{\beta}{2}} P_t w\|_{0} =  c\| A^{\frac{1+\delta+\beta}{2} }P_t A^{-\frac{\beta}{2}} w \|_{0}.
 \end{align*}
Since $w\in \tilde H^{-\beta}(D)$ then by (\ref{bound for the isomorphic spaces H}) we have also $A^{-\frac{\beta}{2}} w\in L_2(D)$ and Theorem \ref{theor pazy} part  (c) ensures that the following bound holds for all $t>0$
\[
\| A^{\frac{1+\delta+\beta}{2} }P_t \|_{\mathscr L(L_2(D))} \leq M e^{-\theta t} t^{-\frac{1+\delta+\beta}{2}}.
\]
This fact together with the previous bound implies
\begin{align*}
 \|P_t w\|_{1+\delta}& \leq c t^{-\frac{1+\delta+\beta}{2}} \|A^{-\frac{\beta}{2}} w\|_0  \leq  c t^{-\frac{1+\delta+\beta}{2}} \| w\|_{-\beta} <\infty,
 \end{align*}
 having used in the last inequality again equation (\ref{bound for the isomorphic spaces H}).
\end{proof}

These two properties can be generalized to a wider range of parameters $\delta$ and $\beta$ (for more details see \cite{t3}).\\
The following integral bounds will be used later. The proof makes use of the Gamma and the Beta functions together with some basic integral estimates.
\begin{lemma}\label{lemma stima integrale con fz gamma}
If $0\leq s<t\leq T<\infty$ and $0\leq \theta<1$ then for any $\rho\geq 1$ it holds
\begin{equation}\label{eq 1 stima integrale con fz gamma}
 \int_s^t e^{-\rho r} r^{-\theta} \mathrm{d} r\leq\Gamma(1-\theta) \rho^{\theta-1}.
\end{equation}
Moreover if $\gamma>0$ is such that $ \theta+\gamma<1$ then for any $ \rho \geq 1$ there exists a positive constant $C$ such that
\begin{equation}\label{eq 2 stima integrale con fz gamma}
\int_0^{t} e^{-\rho(t- r)} (t-r)^{-\theta} r^{-\gamma} \mathrm{d} r\leq C \rho^{\theta -1+\gamma}.
\end{equation}
\end{lemma}

In what follows we state and give the proof of the main mapping property of the integral operator: It is a contraction on a Banach space of function with H\"older-type regularity in time and fractional Sobolev-type regularity in space.

\begin{theorem}
\label{teorema per I in spazi di holder (rho)}
 Let $0<\beta<\delta<\frac{1}{2}$ and  $Z\in H_q^{1-\beta}(\mathbb{R}^d)$ for $q>\max(2, \frac{d}{\delta})$. Then for any $\gamma$ such that $ 0<2\gamma< 1-\delta-\beta$ it holds
\[
I:C^\gamma([0,T];\tilde H^{1+\delta} (D) )\to C^\gamma([0,T];\tilde H^{1+\delta} (D))
\]
and the following estimate holds for any fixed $\underbar {\textit{u}}\in C^\gamma([0,T];\tilde H^{1+\delta} (D))$
\begin{equation}\label{eq mapping property per I (rho)}
 \|I_{(\cdot)}(\underbar {\textit{u}})\|_{\gamma,1+\delta}^{(\rho)} \leq c(\rho)\|\underbar {\textit{u}}\|^{(\rho)}_{\gamma,1+\delta}
\end{equation}
where $c(\rho)$ is a function of $\rho$ not depending on $\underbar {\textit{u}}$ nor $T$ and such that
\[\lim_{\rho\to \infty}c(\rho)=0.\]
\end{theorem}
 \begin{proof}
Given any $\underbar {\textit{u}}\in C^\gamma([0,T];\tilde H^{1+\delta} (D))$ our goal is to bound
\begin{align}\label{equation thesis norm I (rho)}
 \|I_{(\cdot)}(\underbar {\textit{u}})\|_{\gamma,1+\delta}^{(\rho)}= & \sup_{0\leq t \leq T} \Big(e^{-\rho t} \|I_t(\underbar {\textit{u}}) \|_{1+\delta} + \\
&+e^{-\rho t} \sup_{0\leq s<t } \frac{\|I_t(\underbar {\textit{u}})-I_s(\underbar {\textit{u}})\|_{1+\delta}}{(t-s)^\gamma} \Big) \nonumber  =  :\sup_{0\leq t \leq T} \big(  (A) + (B) \big)
\end{align}
using the $(\rho)$-norm of $\underbar {\textit{u}} $, namely using $ \|\underbar {\textit{u}}\|_{\gamma,1+\delta}^{(\rho)}$. \\
\emph{Step 1:} Consider part $(A)$. Fix $t\in [0,T]$.
\begin{align*}
e^{-\rho t} \|I_t(\underbar {\textit{u}})\|_{1+\delta} & = e^{-\rho t} \|\int_0^t P_{t-r} m(r) \, \mathrm{d} r \|_{1+\delta}\\
&\leq e^{-\rho t} \int_0^t \| P_{t-r} m(r) \|_{1+\delta} \, \mathrm{d} r
\end{align*}
apply Proposition \ref{propr su P_t} with $w=m(s)\in\tilde H^{-\beta} (D) $ and afterwards Proposition \ref{proposition m(r) in spazi di sobolev}  (1) and obtain
\begin{align*}
(A)&\leq e^{-\rho t} \int_0^t \|m(r)\|_{-\beta} (t-r)^{-\frac{1+\delta+\beta}{2}} \, \mathrm{d} r\\
&\leq  c e^{-\rho t} \int_0^t \|\underbar {\textit{u}}(r) \|_{1+\delta} (t-r)^{-\frac{1+\delta+\beta}{2}} \, \mathrm{d} r.
\end{align*}
 Observe that for any $0\leq r\leq t\leq T$
\begin{align*}
e^{-\rho r} \|\underbar {\textit{u}}(r)\|_{1+\delta}&\leq \sup_{0\leq r\leq T} e^{-\rho r} \|\underbar {\textit{u}}(r) \|_{1+\delta} \leq  \|\underbar {\textit{u}} \|_{\gamma,1+\delta}^{(\rho)}
\end{align*}
and then we obtain
\begin{align*}
{(A)}&=e^{-\rho t} \|I_t(\underbar {\textit{u}})\|_{1+\delta}  \leq c \|\underbar {\textit{u}} \|_{\gamma,1+\delta}^{(\rho)} \int_0^t  e^{-\rho( t-r) }  (t-r)^{-\frac{1+\delta+\beta}{2}} \, \mathrm{d} r\\
& = c \|\underbar {\textit{u}} \|_{\gamma,1+\delta}^{(\rho)} \int_0^t  e^{-\rho r }  r^{-\frac{1+\delta+\beta}{2}} \, \mathrm{d} r
 \leq c \|\underbar {\textit{u}} \|_{\gamma,1+\delta}^{(\rho)} \rho^{\frac{1+\delta+\beta}{2}-1}
\end{align*}
having used estimate (\ref{eq 1 stima integrale con fz gamma}) of Lemma \ref{lemma stima integrale con fz gamma} in the last line. Clipping the result together we can state that
\begin{equation*}\label{eq bound for part A (rho)}
 (A)=e^{-\rho t} \|I_t(\underbar {\textit{u}})\vert \tilde H^{1+\delta} (D)\| \leq c( \rho) \|\underbar {\textit{u}} \|_{\gamma,1+\delta}^{(\rho)}
\end{equation*}
 where $c(\rho)= c \rho^{\frac{\delta+\beta-1}{2}} $ and since $\frac{\delta+\beta-1}{2}<0$ we have $c(\rho)\to0$ as $\rho\to\infty$. \\
\emph{Step 2:} Consider part $(B)$. Let for the moment fix our attention only on the argument inside the norm in the numerator of $(B)$. Recall that $0\leq s<t\leq T$. We make a change of variable in the middle integral $r'=r-t+s$ and we obtain
\begin{align*}
  & \int_0^t P_{t-r} m(r) \, \mathrm{d} r -\int_0^s P_{s-r} m(r) \, \mathrm{d} r\\
 =&\int_0^{t-s} P_{t-r} m(r) \, \mathrm{d} r + \int_{t-s}^t P_{t-r} m(r) \, \mathrm{d} r  -\int_0^s P_{s-r} m(r) \, \mathrm{d} r\\
 =&\int_0^{t-s} P_{t-r} m(r) \, \mathrm{d} r + \int_0^{s} P_{s-r} m(r+t-s) \, \mathrm{d} r - \int_0^{s} P_{s-r} m(r) \, \mathrm{d} r \\
 =&\int_0^{t-s} P_{t-r} m(r) \, \mathrm{d} r + \int_0^{s} P_{s-r} (m(r+t-s)-m(r) ) \, \mathrm{d} r
\end{align*}

These computations enable us to write
\begin{align*}
{(B)}=& e^{-\rho t}  \sup_{ 0\leq s<t} \frac{\|I_t(\underbar {\textit{u}})-I_s(\underbar {\textit{u}}) \|_{1+\delta}}{(t-s)^\gamma} \\
 = &e^{-\rho t}\sup_{0\leq s<t} \frac{\| \int_0^t P_{t-r} m(r) \, \mathrm{d} r -\int_0^s P_{s-r} m(r) \, \mathrm{d} r \|_{1+\delta}}{(t-s)^\gamma} \\
 \leq & e^{-\rho t}\sup_{0\leq s<t} \frac{\| \int_0^{t-s} P_{t-r} m(r) \, \mathrm{d} r \|_{1+\delta}}{(t-s)^\gamma} \\
& + e^{-\rho t}\sup_{0\leq s<t} \frac{\| \int_0^s P_{s-r}( m(r+t-s)-m(r)) \, \mathrm{d} r \|_{1+\delta}}{(t-s)^\gamma}
 := (C) + (D).
\end{align*}
\emph{Step 3:} Consider term $(C)$.\\
The numerator is similar to the term $(A)$ and therefore we proceed as we did in Step 1. We have
\begin{align*}
{(C)}= & e^{-\rho t}\sup_{0\leq s<t} \frac{\| \int_0^{t-s} P_{t-r} m(r) \, \mathrm{d} r \|_{1+\delta}}{(t-s)^\gamma} \\
\leq & e^{-\rho t}\sup_{0\leq s<t} \frac{ \int_0^{t-s} c  \| \underbar {\textit{u}}(r)   \|_{1+\delta} (t-r)^{-\frac{1+\delta+\beta}{2}}  \, \mathrm{d} r}{(t-s)^\gamma} \\
\leq & \sup_{0\leq s<t}  \int_0^{t-s}   e^{-\rho (t-r)}c \| \underbar {\textit{u}} \|_{\gamma,1+\delta}^{(\rho)} (t-r)^{-\frac{1+\delta+\beta}{2}}   (t-s)^{-\gamma} \, \mathrm{d} r \\
\leq &  c \| \underbar {\textit{u}} \|_{\gamma,1+\delta}^{(\rho)}  \sup_{0\leq s<t} \int_0^{t-s}   e^{-\rho (t-r)} (t-r)^{-\frac{1+\delta+\beta}{2}} r^{-\gamma} \, \mathrm{d} r\\
= &  c \| \underbar {\textit{u}} \|_{\gamma,1+\delta}^{(\rho)}  \int_0^{t}   e^{-\rho (t-r)} (t-r)^{-\frac{1+\delta+\beta}{2}} r^{-\gamma} \, \mathrm{d} r
\end{align*}
apply estimate (\ref{eq 2 stima integrale con fz gamma}) in Lemma \ref{lemma stima integrale con fz gamma} with $\theta=\frac{1+\delta+\beta}{2}$: since by hypothesis $2\gamma<1-\delta-\beta$ then $\gamma+\theta<1$. We obtain
\begin{align*}
{(C)} \leq    c \| \underbar {\textit{u}} \|_{\gamma,1+\delta}^{(\rho)}   \rho^{\frac{1+\delta+\beta+2\gamma}{2}-1} \leq & c \| \underbar {\textit{u}} \|_{\gamma,1+\delta}^{(\rho)}  \rho^{\frac{\delta+\beta+2\gamma-1}{2}}.
\end{align*}

Clipping the result together
\begin{equation*}\label{eq bound for part C (rho)}
{(C)}= e^{-\rho t} \sup_{0\leq s<t} \frac{\| \int_0^{t-s} P_{t-r} m(r) \, \mathrm{d} r \|_{1+\delta}}{(t-s)^\gamma}  \leq c_1 \|\underbar {\textit{u}} \|_{\gamma,1+\delta}^{(\rho)}  \rho^{\frac{\delta+\beta+2\gamma-1}{2}}.
\end{equation*}
\emph{Step 4:} Consider term $(D)$.\\
First apply Proposition \ref{propr su P_t} to $w=m(r+t-s)-m(r)$ which is an element of  $\tilde H^{-\beta} (D)$ thanks to Property \ref{proposition maggiorazione prodotto puntuale per uZ}. Then apply Proposition \ref{proposition m(r) in spazi di sobolev}, (2).
\begin{align*}
&{(D)}= e^{-\rho t}\sup_{0\leq s<t} \frac{\| \int_0^s P_{s-r}( m(r+t-s)-m(r)) \, \mathrm{d} r  \|_{1+\delta}}{(t-s)^\gamma} \\
\leq & e^{-\rho t}\sup_{0\leq s<t} \frac{ \int_0^s \|  m(r+t-s)-m(r)  \|_{-\beta} (s-r)^{-\frac{1+\delta+\beta}{2}} } {(t-s)^\gamma}  \, \mathrm{d} r\\
\leq &  c e^{-\rho t}\sup_{0\leq s<t} \int_0^s  \frac{ e^{-\rho (r+t-s)}}{ e^{-\rho (r+t-s)}} \frac{ \|  \underbar {\textit{u}}(r+t-s)-\underbar {\textit{u}}(r) \|_{1+\delta} (s-r)^{- \frac{1+\delta+\beta}{2}} } {(t-s)^\gamma} \, \mathrm{d} r \\
\leq &  c \sup_{0\leq s<t} \int_0^s e^{-\rho (s-r)}  e^{-\rho (r+t-s)} \frac{ \|  \underbar {\textit{u}}(r+t-s)-\underbar {\textit{u}}(r) \|_{1+\delta} } {(t-s)^\gamma}  (s-r)^{- \frac{1+\delta+\beta}{2}} \, \mathrm{d} r.
\end{align*}
Fix the attention on the term $ e^{-\rho (r+t-s)} \frac{ \|  \underbar {\textit{u}}(r+t-s)-\underbar {\textit{u}}(r)  \|_{1+\delta}  } {(t-s)^\gamma}$ and set $h=t-s$: we obtain
\begin{equation}\label{eq norma di u riscritta con h=t-s}
  e^{-\rho (r+h)} \frac{ \|  \underbar {\textit{u}}(r+h)-\underbar {\textit{u}}(r) \|_{1+\delta} } {h^\gamma}.
\end{equation}
Moreover observe that
\[
\|\underbar {\textit{u}}\|^{(\rho)}_{\gamma,1+\delta}= \sup_{0\leq t\leq T} e^{-\rho t} \|\underbar {\textit{u}}(t)\|_{1+\delta} +  \sup_{0\leq r<t \leq T}  e^{-\rho t} \frac{\|\underbar {\textit{u}}(t)-\underbar {\textit{u}}(r)\|_{1+\delta}}{(t-r)^\gamma}
\]
and in particular, setting again $t-r=h$, the second summand can be rewritten as
\[
  \sup_{0<h\leq r+h\leq T } e^{-\rho (r+h)} \frac{\|  \underbar {\textit{u}}(r+h)-\underbar {\textit{u}}(r)   \|_{1+\delta}}{h^\gamma}.
\]
Therefore we can bound (\ref{eq norma di u riscritta con h=t-s}) by $\|\underbar {\textit{u}}\|_{\gamma,1+\delta}^{(\rho)}$ (since the parameters $r$ and $h$ are such that $0<h\leq r+h\leq T$) and applying once more estimate (\ref{eq 1 stima integrale con fz gamma}) in Lemma \ref{lemma stima integrale con fz gamma} the upper bound for $(D)$ becomes
\begin{align*}
{(D)}\leq  &   c \|\underbar {\textit{u}}\|_{\gamma,1+\delta}^{(\rho)} \sup_{0\leq s<t} \int_0^s e^{-\rho (s-r)}  (s-r)^{- \frac{1+\delta+\beta}{2}} \, \mathrm{d} r\\
\leq &  c_2 \|\underbar {\textit{u}}\|_{\gamma,1+\delta}^{(\rho)} \rho^{\frac{\delta+\beta-1}{2}} \Gamma\left(\frac{\delta+\beta-1}{2}\right).
\end{align*}
Clipping the result for part $(B)$ we obtain

\begin{equation}\label{eq bound for part B (rho)}
{(B)}=(C)+(D)= e^{-\rho t}  \sup_{ 0\leq s<t} \frac{\|I_t(\underbar {\textit{u}})-I_s(\underbar {\textit{u}}) \|_{1+\delta}}{(t-s)^\gamma} \leq c'( \rho) \|\underbar {\textit{u}} \|_{\gamma,1+\delta}^{(\rho)}
\end{equation}
 where $c'(\rho)= c_1 \rho^{\frac{\delta+\beta+2\gamma-1}{2}}+ c_2  \rho^{\frac{\delta+\beta-1}{2}}$ and since $\frac{\delta+\beta+2\gamma-1}{2}$ and $\frac{\delta+\beta-1}{2}$ are negative we have $c'(\rho)\to0$ as $\rho\to\infty$.

Finally observe that the bound for $(A)+(B)$ does not depend on $t$ and then the supremum over $0\leq t\leq T$ of  $(A)+(B)$ is simply bounded by
\[
 \|I_{(\cdot)}(\underbar {\textit{u}})\|_{\gamma,1+\delta}^{(\rho)}=\sup_{0\leq t\leq T} \Big( (A)+(B) \Big) \leq (c(\rho)+c'(\rho)) \|\underbar {\textit{u}} \|_{\gamma,1+\delta}^{(\rho)}
\]
that is the thesis.
 \end{proof}

\subsection{Theorem of existence and uniqueness}
Now we prove existence and uniqueness of a global mild solution.
\begin{theorem}\label{teo esistenza e unicita globale}
Let $0<\beta<\delta<\frac{1}{2}$ and $ 0<2\gamma< 1-\beta-\delta$. Fix $Z\in H_q^{1-\beta}(\mathbb{R}^d)$ for some $q>\max(2, \frac{d}{\delta})$. Then for any initial condition $u_0\in\tilde H^{1+\delta+2\gamma}(D)$ and for any positive finite time $T$ there exists a unique mild solution $\underbar {\textit{u}}$ in $C^\gamma ([0,T];\tilde H^{1+\delta}(D) )$ for (\ref{SPDE Cauchy prb}) satisfying the integral equation $\underbar {\textit{u}}(t)=P_t u_0+I_t(\underbar {\textit{u}})$.
\end{theorem}

\begin{proof}
From Theorem \ref{teorema per I in spazi di holder (rho)} we know that if $\underbar {\textit{u}}\in C^\gamma([0,T];\tilde H^{1+\delta} (D)) $ then $I_{(\cdot)}(\underbar {\textit{u}}) \in C^\gamma([0,T];\tilde H^{1+\delta} (D)) $.\\
 Now we should ensure that for  $u_0\in\tilde H^{\sigma}(D)$, with $\sigma\geq 1+\delta+2\gamma$ then  $P_{(\cdot)} u_0 \in  C^\gamma([0,T];\tilde H^{1+\delta} (D)) $ too, using the $(\rho)$-norm, namely we should check that
\[
\sup_{0\leq t \leq T} e^{-\rho t}\left( \|P_tu_0\|_{1+\delta} + \sup_{0\leq s<t} \frac{\|P_tu_0-P_su_0\|_{1+\delta} }{(t-s)^\gamma}\right) < \infty.
\]
 Recall that $P_t$ is a bounded linear operator on $\tilde H^\sigma(D)$  for  $-\frac{1}{2}<\sigma$, therefore for every $x\in \tilde H^\sigma(D) $, $\|P_t x \vert  \tilde H^\sigma(D)\|\leq \|P_t\|\cdot \|x \|_{\sigma}<\infty$. Since $u_0\in  \tilde H^{1+\delta+2\gamma}(D) \subset  \tilde H^{1+\delta}(D)$ then
 \[
 \sup_{0\leq t \leq T}  e^{-\rho t} \|P_tu_0\|_{1+\delta} \leq c \sup_{0\leq t \leq T}  \|P_tu_0\|_{1+\delta}  <\infty.
 \]
For the second summand use part (d) of Theorem \ref{theor pazy} and relation (\ref{bound for the isomorphic spaces H}) to obtain
\begin{align*}\label{eq regolarita u0}
 & \|P_tu_0-P_su_0\|_{1+\delta}\nonumber
 = \|P_s(P_{t-s}-I)u_0\|_{1+\delta} \nonumber \\
&\leq  c \|P_s\| \| (P_{t-s}-I) u_0\|_{1+\delta}  \nonumber
\leq  c \|P_s\| (t-s)^\alpha \| A^\alpha u_0\|_{1+\delta} \\
&\leq  c \|P_s\| (t-s)^\alpha \|  u_0\|_{1+\delta +2\alpha }  \nonumber
\leq c M e^{-\omega s} (t-s)^\alpha \|  u_0\|_{1+\delta +2\alpha} \nonumber
\end{align*}
for any $0<\alpha<1$. Therefore the second summand becomes
\begin{align*}
 &\sup_{0\leq t\leq T} e^{-\rho t} \sup_{0\leq s<t} \frac{\|P_tu_0-P_su_0\|_{1+\delta}}{(t-s)^\gamma}
\leq  \sup_{0\leq s< t\leq T} e^{-\rho t} c_s (t-s)^\alpha \frac{ \|  u_0\|_{1+\delta +2\alpha}}{(t-s)^\gamma}
\end{align*}
and if we choose $\alpha=\gamma$ then
\[
\sup_{0\leq t\leq T} e^{-\rho t} \sup_{0\leq s<t} \frac{\|P_tu_0-P_su_0\|_{1+\delta}}{(t-s)^\gamma} \leq \sup_{0\leq s < t\leq T} e^{-\rho t} c_s  \|  u_0\|_{1+\delta +2\gamma}
\]
that is a finite  quantity if $u_0 \in \tilde H^{1+\delta +2\gamma}(D)$.\\
So for any fixed  $u_0 \in \tilde H^{1+\delta+2\gamma}(D)$  the operator $J_{(\cdot)}:=P_{(\cdot)} u_0+I_{(\cdot)}$ is mapping $C^\gamma([0,T];\tilde H^{1+\delta} (D))$ into itself. It is left to prove that $J_{(\cdot)}$ is a contraction, namely that there exists a constant $k<1$ such that for all $\underbar {\textit{u}}, \underbar {\textit{v}} \in C^\gamma([0,T];\tilde H^{1+\delta} (D))$
 \[
  \|J_{(\cdot)}(\underbar {\textit{u}})-J_{(\cdot)}(\underbar {\textit{v}})\|_{\gamma,1+\delta}^{(\rho)}\leq k \|\underbar {\textit{u}} - \underbar {\textit{v}} \|^{(\rho)}_{\gamma,1+\delta}.
 \]
For this aim observe that
\begin{align*}
  \|J_{(\cdot)}(\underbar {\textit{u}})-&J_{(\cdot)}(\underbar {\textit{v}})\|_{\gamma,1+\delta}^{(\rho)}= \|P_tu_0+ I_{(\cdot)}(\underbar {\textit{u}})-P_tu_0-I_{(\cdot)}(\underbar {\textit{v}})\|_{\gamma,1+\delta}^{(\rho)}\\
  &= \left\|\int_0^{\cdot}P_{\cdot-r}  \langle \nabla \underbar {\textit{u}}(r), \nabla Z\rangle \mathrm{d} r- \int_0^{\cdot}P_{\cdot-r}  \langle \nabla \underbar {\textit{v}}(r), \nabla Z\rangle \mathrm{d} r   \right\|_{\gamma,1+\delta}^{(\rho)}\\
  &\leq \left\| \int_0^{\cdot}P_{\cdot-r} \left(\langle \nabla ( \underbar {\textit{u}}(r)-\underbar {\textit{v}}(r) ), \nabla Z\rangle \mathrm{d} r \right)   \right\|_{\gamma,1+\delta}^{(\rho)}\leq  \| I_{(\cdot)}(\underbar {\textit{u}}-\underbar {\textit{v}})\|_{\gamma,1+\delta}^{(\rho)}.
 \end{align*}
We clearly have $\underbar {\textit{w}}:=\underbar {\textit{u}}-\underbar {\textit{v}}\in C^\gamma([0,T];\tilde H^{1+\delta} (D)) $ and then it suffices to apply the result of Theorem \ref{teorema per I in spazi di holder (rho)} with $\underbar {\textit{w}}$ instead of $\underbar {\textit{u}}$ and choose $\rho$ big enough such that the constant $c(\rho)$ appearing in (\ref{eq mapping property per I (rho)}) is less than 1.
\end{proof}


\section{Applications}\label{section applications}
In this section we will  apply the previous results to some stochastic PDEs.
\subsection{The stochastic transport equation}\label{subsection stochastic transport equation}
Consider the stochastic transport equation given by

\begin{equation}\label{stochastic PDE}
\left\{\begin{array}{lcr}
\frac{\partial u}{\partial t}(t,x)= \sigma^2 \Delta u (t,x) + \langle \nabla u (t,x), \nabla Y (x,\omega) \rangle,&\quad& t   \in (0,T], x\in D\\
u(t,x)=0, && t\in (0,T], x\in \partial D\\
u(0,x)=u_0(x), && x\in D
\end{array}\right.
\end{equation}
 where $Y=\{Y(x,\omega)\}_{x\in \mathbb{R}^d}$ is a stochastic field defined on a given probability space $(\Omega, \mathcal F,\mathbb{P})$. One suitable example for the noise $Y$ is the \emph{Levy fractional Brownian motion} $\{B^H(x)\}_{x\in\mathbb{R}^d}$ which is the isotropic generalization of the fractional Brownian motion (see \cite{l1}). This field is defined to be a centered Gaussian field on $\mathbb{R}^d$ of covariance function
\[
 \mathbb{E}[B^H(x)B^H(y)]=\frac{1}{2} (\vert x\vert_d^{2H} +\vert y\vert_d^{2H} -\vert x-y\vert_d^{2H} )
\]
where $\vert\cdot\vert_d$ stands for the Euclidean norm in $\mathbb{R}^d$. The parameter $0<H<1$ is called Hurst parameter. In case when $H=\frac{1}{2}$ we recover the \emph{Levy Brownian motion}, whereas if $d=1$ we get the fractional Brownian motion. Using a Kolmogorov continuity theorem suitable for stochastic fields (see for instance \cite{k2}, Theorem 1.4.1) and basic properties of Gaussian random variables one can show that there exist $\Omega_1\subset \Omega$ with $\mathbb{P}(\Omega_1)=1$ and a modification of $B^H(x), x\in D$ (for simplicity we call it again $B^H(x)$) with $D\subset\mathbb{R}^d$ arbitrary bounded domain of $\mathbb{R}^d$ such that for every $\omega \in \Omega_1$ and for every $x,y\in D$ we have
\[
\vert B^H(x,\omega)-B^H( y,\omega)\vert\leq K_\omega \vert x-y\vert_d^\alpha, \: \forall \alpha<H
\]
where $K$ is a positive random variable with finite moments of every order. \\
In other words, for almost every realization $\omega$ the field is $\alpha$-H\"older continuous on $D$ of any order  $\alpha<H$. This fact together with the following property enable us to apply the results presented in the previous section to equation (\ref{stochastic PDE}) in a pathwise sense.

\begin{proposition}\label{proposition holder continuity -> sobolev regularity}
 Let $h$ be a compactly supported real valued $\alpha$-H\"older continuous function on $\mathbb{R}^d$ for some $0<\alpha<1$. Then for any $\alpha'<\alpha$ we have $h\in H^{\alpha'}_p(\mathbb{R}^d)$ for all $2\leq p<\infty$.
\end{proposition}
The proof makes use of the equivalent norm
\[
\|h\|_{L^p}+\left(\int_{\vert y\vert \leq 1}\frac{\|h(\cdot+y)-h(\cdot)\|_{L^p}^2}{\vert y\vert^{d+2\alpha'}}\mathrm{d} y\right)^{\frac{1}{2}}
\]
 for the Besov spaces $B^{\alpha'}_{p,2}(\mathbb{R}^d)$ and of embedding properties between Besov and Sobolev spaces (see \cite{t2} for more details).

In order to apply this to (almost every) path of $B^H$ we should ensure the compactness of the support. This is not true in general. Instead, since (\ref{stochastic PDE}) is considered only on the  domain $D$,  let $\psi(x), x\in \mathbb{R}^d$ be a $C^\infty$-function  with compact support and such that $\psi(x)=1 \ \forall  x\in \bar D$. Then for almost every $\omega\in \Omega$ the function  $\psi(x) B^H(\omega,x)$ is $\alpha$-H\"older continuous we have that for all $1<q<\infty $ and for all $ \alpha'<\alpha<H$, $\psi(\cdot)B^H(\omega, \cdot)\in H_q^{\alpha'} (\mathbb{R}^d)$. For consistency of notation call $1-\beta:=\alpha'$, and so $1-\beta<H$. In order to match the conditions on the parameter $\beta$ we have to choose $\frac{1}{2}<H<1$. Then for every $\omega\in \Omega_1$ we set $Z(x):=\psi(x)B^H(\omega, x)$ and so Theorem \ref{teo esistenza e unicita globale} ensures existence and uniqueness of a function solution to the stochastic Dirichlet problem (\ref{stochastic PDE}) with $Y=B^H$.

\subsection{A (more) general stochastic transport equation}
We combine in this section the main result obtained in this paper with a result obtained in \cite{h1}. \\
Recall Definition 2.1 in \cite{h1} (we only need the case $k=1$) where the authors define an integral operator of the type $I^\alpha_t(F,\frac{\partial}{\partial t} \nabla V)$ for some  given $F\in \mathbb{R}^d$ and $V=V(t,x_1,\ldots,x_d)$. Their idea is to use Fourier transform to perform the integration with respect to the space variable $x$ and fractional derivatives to give a meaning to the derivative with respect to time and then perform the integration. Moreover they exploited the regularity of this integral, and they proved in Proposition 7.1 that
if $0<\alpha,\beta,\gamma<1$ with $\alpha+\gamma<1$ and $2\gamma+\tilde\delta<2-2\alpha-\beta$ then the integral $I^\alpha_{(\cdot)}(F,\frac{\partial}{\partial t} \nabla V)$ (which in fact does not depend on $\alpha$) belongs to the space $C^\gamma ([0,T];\tilde H^{\tilde \delta}(D) )$ for any given function $V\in C^{1-\alpha}([0,T];H^{1-\beta}(\mathbb{R}^d))$ and vector $F\in  \mathbb{R}^d$.

Taking this into account we are able to give the following  existence and uniqueness result.

\begin{corollary}
Let $T>0$ be fixed, choose $0<\beta<\delta<\frac{1}{2}$ and $ 0<2\gamma< 1-\beta-\delta$. Fix $F\in\mathbb{R}^d$,  $Z\in H_q^{1-\beta}(\mathbb{R}^d)$ and $V\in C^{1-\alpha}([0,T]; H_q^{1-\beta}(\mathbb{R}^d))$ for some $q>\max(2,\frac{d}{\delta})$ and for some $0<\alpha<1$ such that $\alpha+\gamma<1$. Then given any initial condition $u_0\in\tilde H^{1+\delta+2\gamma}(D)$ there exists a unique global mild solution $u(t,x)$ in the  H\"older space $C^\gamma ([0,T];\tilde H^{1+\delta}(D) )$ for the problem
\begin{equation}\label{general PDE}
\left\{\begin{array}{lr}
\frac{\partial u}{\partial t}(t,x)= \sigma^2 \Delta u (t,x) + \langle \nabla u (t,x), \nabla Z(x) \rangle \  & \\
 \phantom{\frac{\partial u}{\partial t}(t,x)= \sigma^2 \Delta u (t,x) } +\langle F ,\frac{\partial}{\partial t} \nabla V(t,x) \rangle,&\   t   \in (0,T], x\in D\\
u(t,x)=0, &  t\in (0,T], x\in \partial D\\
u(0,x)=u_0(x), & x\in D
\end{array}\right.
\end{equation}
and the solution is given by
\[
u(t,\cdot)=P_t u_0+I_t(\underbar {\textit{u}})+I^\alpha_t(F,\frac{\partial}{\partial t} \nabla V).
\]
\end{corollary}

\begin{proof}
Set $\tilde\delta:=1+\delta$. Since $2\gamma<1-\delta-\beta$ then $2\gamma+\tilde \delta<2-\beta$ and if one chooses a positive $\alpha$ such that $2\gamma+\tilde \delta<2-\beta-2\alpha$ then the condition $\alpha+\gamma<1$ is satisfied and by Proposition 7.1 in \cite{h1} we have $I^\alpha_{(\cdot)}(F,\frac{\partial}{\partial t} \nabla V) \in C^\gamma ([0,T];\tilde H^{\tilde \delta}(D) )$. Finally apply a contraction principle  as applied in the proof of Theorem \ref{teo esistenza e unicita globale} and recover the thesis.
\end{proof}
With the same technique illustrated in Section \ref{subsection stochastic transport equation} one can solve (\ref{general PDE}) in the case when  $Z$ and $ V$ are substituted by stochastic fields, and then the system is solved in the pathwise sense. See \cite{h1}, Section 6 for a survey on possible noises in place of $V$.


\section*{Acknowledgements}
Work supported in part by the European Community's FP 7 Programme under contract PITN-GA-2008-213841, Marie Curie ITN ``Controlled Systems''.


\begin{thebibliography}{99}



\bibitem{c1}  Chow,~P.-L., \emph{Stochastic partial differential equations}, Chapman \& Hall/CRC Applied Mathematics and Nonlinear Science Series, Boca Raton, FL, 2007.

\bibitem{d1}  Da\ Prato,~G. and  Zabzcyk,~J., \emph{Stochastic equations in infinite dimensions}, Cambridge Univ. Press, Cambridge 1992.

\bibitem{d2} Davies,~E.B., \emph{Heat Kernels and Spectral Theory}, Cambridge Univ. Press, Cambridge, 1989.

\bibitem{d3} Duncan,~T. E., Maslowski,~B. and Pasik-Duncan,~B., Solutions of linear and semilinear distributed parameter equations with a fractional Brownian motion.  Internat. J. Adapt. Control Signal Process.  23  (2009),  no. 2, 114-–130.

\bibitem{e1}  Evans,~L., \emph{Partial differential equations}, GSM 19, American Mathematical Society, Providence, RI, 1998.

\bibitem{g1} Grecksch,~W. and Anh,~V.~V., A parabolic stochastic differential equation with fractional Brownian motion input.  Statist. Probab. Lett.  41  (1999),  no. 4, 337–-346.

\bibitem{h1} Hinz,~M. and Z\"ahle,~M., Gradient type noises. II. Systems of stochastic partial differential equations.  J. Funct. Anal.  256  (2009),  no. 10, 3192-–3235.

\bibitem{h2} Holden,~H., {\O}ksendal,~B., Ub{\o}e,~J. and Zhang,~T., \emph{Stochastic partial differential equations} A modeling, white noise functional approach., Birkh\"auser Boston Inc., Boston MA,  1996.

\bibitem{h3} Hu,~Y.,  A class of SPDE driven by fractional white noise.  Stochastic processes, physics and geometry: new interplays, II (Leipzig, 1999),  317-–325, CMS Conf. Proc., 29, Amer. Math. Soc., Providence, RI, 2000.

\bibitem{k2}  Kunita,~H., \emph{Stochastic flows and stochastic differential equations}, Cambridge University Press, Cambridge, 1990.

\bibitem{l1}  Lindstr{\o}m,~T., Fractional Brownian fields as integrals of white noise. Bull. London Math. Soc. 25 (1993), no. 1, 83–-88.

\bibitem{l2}  Lunardi,~A., \emph{Analytic semigroups and optimal regularity in parabolic problems}, Birkh\"auser Verlag, Basel, 1995.

\bibitem{m1} Maslowski,~B. and Nualart,~D., Evolution equations driven by a fractional Brownian motion. J. Funct. Anal.  202  (2003),  no. 1, 277-–305.

\bibitem{p1}  Pazy,~A., \emph{Semigroups of Linear Operators and Applications to Partial Differential Equations}, Springer, New York, 1983.

\bibitem{r1}  Runst,~T. and  Sickel,~W. \emph{Sobolev Spaces of Fractional Order, Nemytskij Operators, and nonlinear Partial Differential Equations}, de Gruyter Series in Nonlinear Analysis and Applications, 3. Walter de Gruyter \& Co., Berlin, 1996.

\bibitem{r2}  Russo,~F. and Trutnau,~G., Some parabolic PDEs whose drift is an irregular random noise in space.  (English summary) Ann. Probab. 35 (2007), no. 6, 2213–-2262.

\bibitem{s1} Stein,~Elias M., \emph{Topics in Harmonic Analysis Related to the Littlewood-Paley Theory} Annals of Mathematics Studies, No. 63 Princeton University Press, Princeton, N.J.; University of Tokyo Press, Tokyo 1970.

\bibitem{t1} Tindel,~S., Tudor,~C. A. and Viens,~F.,  Stochastic evolution equations with fractional Brownian motion.  Probab. Theory Related Fields  127  (2003),  no. 2, 186–-204.

\bibitem{t2} Triebel,~H., \emph{Theory of function spaces II}, Monographs in Mathematics, 84. Birkhäuser Verlag, Basel, 1992.

\bibitem{t3} Triebel,~H.,  \emph{Interpolation Theory, Function Spaces, Differential Operators}, Second edition. Johann Ambrosius Barth, Heidelberg, 1995.

\bibitem{t4}  Triebel,~H., \emph{Fractals and spectra. Related to Fourier analysis and function spaces.}, Monographs in Mathematics, 91.  Birkh\"auser Verlag, Basel, 1997.

\bibitem{v1}    Vrabie,~Ioan I.,  \emph{$C_0$-semigroups and applications} North-Holland Mathematics Studies, 191. North-Holland Publishing Co., Amsterdam, 2003.

\end{thebibliography}
\end{document}